\documentclass{article}

\usepackage{endnotes}
\usepackage[utf8]{inputenc}

\usepackage{authblk}
\usepackage{amsfonts}
\usepackage{amsmath}
\usepackage{amsthm}
\usepackage{amssymb}
\usepackage{graphicx}
\usepackage{amscd}
\usepackage{enumerate}
\usepackage{color}
\usepackage{verbatim}
\usepackage{appendix}
\usepackage{url}
\usepackage{cleveref}
\usepackage{array}

\newtheorem{definition}{Definition}[section]
\newtheorem{example}{Example}[section]
\newtheorem{lemma}{Lemma}[section]
\newtheorem{theorem}{Theorem}[section]

\newcommand{\ignore}[1]{}

\newcommand{\f}{\mathit{f}}

\newcolumntype{C}{>{$}c<{$}}
\newcolumntype{L}{>{$}l<{$}}
\newcolumntype{R}{>{$}r<{$}}

\usepackage{tikz}
\usepackage{tikz-timing}
\usetikzlibrary{positioning}
\usepackage{tikz-qtree}
\usepackage{calc}
\usetikzlibrary{arrows.meta}

\newcommand\fano{
        \draw[ultra thick,color=orange] (4,{4*sqrt(3)/3}) circle ({4*sqrt(3)/3}) {};
        \node[draw,circle] (n3) at (0,0) {$3$};
        \node[draw,circle,fill=white] (n5) at (4,0) {$5$};
        \node[draw,circle] (n6) at (8,0) {$6$};
        \node[draw,circle,fill=white] (n2) at (2,{2*sqrt(3)}) {$2$};
        \node[draw,circle,fill=white] (n7) at (6,{2*sqrt(3)}) {$7$};
        \node[draw,circle] (n1) at (4,{4*sqrt(3)}) {$1$};
        \node[draw,circle] (n4) at (4,{4*sqrt(3)/3}) {$4$};
        
        \draw[ultra thick,color=blue] (n3) -- (n5) -- (n6);
        \draw[ultra thick,color=purple] (n2) -- (n4) -- (n6);
        \draw[ultra thick,color=green!50] (n3) -- (n2) -- (n1);
        \draw[ultra thick,color=blue!50] (n1) -- (n7) -- (n6);
        \draw[ultra thick,color=green] (n3) -- (n4) -- (n7);
        \draw[ultra thick,color=red] (n1) -- (n4) -- (n5);
}

\newcommand\brackets[3]{
    \foreach \x [count=\i] in {#3} {
        \node[draw,circle,label={270:$\x$}] (n\i) at ({#2*\i},0) {};
    }         
        \node[circle,fill=black] (n10) at ({1.5*#2},1) {};
        \node[circle,fill=black] (n11) at ({3.5*#2},1) {};
        \node[circle,fill=black] (n12) at ({5.5*#2},1) {};
        \node[circle,fill=black] (n13) at ({7.5*#2},1) {};
        \node[circle,fill=black] (n20) at ({2.5*#2},2.5) {};
        \node[circle,fill=black] (n21) at ({6.5*#2},2.5) {};
        \node[circle,fill=black] (n30) at ({4.5*#2},4.5) {};

        \foreach \x/\y/\z in {0/1/2,1/3/4,2/5/6,3/7/8} {
            \draw[thick, color=gray] (n1\x) -- (n\y);
            \draw[thick, color=gray] (n1\x) -- (n\z);
        }
        \foreach \x/\y/\z in {0/0/1,1/2/3} {
            \draw[densely dashed, color=green!50!red] (n2\x) -- (n1\y);
            \draw[densely dashed, color=green!50!red] (n2\x) -- (n1\z);
        }
        \draw[loosely dashed, color=black] (n30) -- (n20);
        \draw[loosely dashed, color=black] (n30) -- (n21);

}

\usepackage[style=authoryear]{biblatex}
\bibliography{bib}

\begin{document}

\date{November 2022}

\title{How to Design a Stable \\ Serial Knockout Competition}

\author[1]{Roel Lambers\thanks{r.lambers@tue.nl}}
\author[1]{Rudi Pendavingh}
\author[1]{Frits Spieksma}

\affil[1]{Department of Mathematics \& Computer Science, Eindhoven University of Technology, Eindhoven, Netherlands}

\maketitle

\begin{abstract}
We investigate a new tournament format that consists of a series of individual knockout tournaments; we call this new format a Serial Knockout Competition (SKC). This format has recently been adopted by the Professional Darts Corporation. Depending on the seedings of the players used for each of the knockout tournaments, players can meet in the various rounds (eg first round, second round, ..., semi-final, final) of the knockout tournaments. Following a fairness principle of treating all players equal, we identify an attractive property of an SKC: each pair of players should potentially meet equally often in each of the rounds of the SKC. If the seedings are such that this property is indeed present, we call the resulting SKC {\em stable}. In this note we formalize this notion, and we address the question: do there exist seedings for each of the knockout tournaments such that the resulting SKC is stable?

We show, using a connection to the Fano plane, that the answer is yes for 8 players. We show how to generalize this to any number of players that is a power of 2, and we provide stable schedules for competitions on 16 and 32 players.
\end{abstract}
\newpage

\section{Introduction}
\label{sec:intro}
Two popular tournament formats are the round robin format and the knockout format. In a round robin format, each pair of players (or teams) meet a given number of times. 
In a knockout tournament, starting from a so-called {\em seeding}, each round of the knockout tournament sees matches between all remaining players, and a player is removed from the tournament after losing a match; in this way, after $\mbox{log } n$ rounds a winner is determined (where $n$ is the number of players). 

Each of these formats has been studied intensely from very different viewpoints. In particular, deciding upon a seeding of the players in a single knockout tournament has attracted a lot of attention; we do not aim to review this field, and simply refer to \cite{Ho+Ri1985}, \cite{Vu2010}, \cite{Vu+Shoham2011}, \cite{Groh2012}, \cite{Aziz2014}, \cite{Karpov2016}, \cite{Pa+Su2022}, and the references contained therein for more information on this subject. Most of this literature assumes that probabilities are given that denote the chance of one player beating the other. 

In practice, it is not uncommon to design a tournament combining both formats: for instance, first have a number of round robin tournaments in parallel, and then let the winners of the round robins participate in a knockout tournament.

In this note we study a new format that can be seen as an alternative combination of a knockout tournament and a round robin tournament. Let the number of players $n$ be equal to $2^{k}$ for some $k \geq 2$, allowing us to focus exclusively on so-called {\em balanced} knockout tournaments, i.e., knockout tournaments where each player has to play the same number of matches to win the tournament. Observe that a balanced knockout tournament consists of $k$ successive {\em rounds}, where in round $i$ the remaining $2^{k+1-i}$ players compete, $i=1, \ldots,k$. 

The competition format we study consists of a set of $2^{k}-1$ knockout tournaments. We will call this format a {\em Serial Knockout Competition}, or SKC for short. Related (but different) formats are the so-called quasi-double knockout tournament (\cite{Considine2018}) and the multiple-elimination knockout tournament (\cite{Fayers2005}). The problem that we analyze in this note is to specify, for each of the individual knockout tournaments that make up the SKC, the {\em seeding}; these seedings specify, for each player, the leaf nodes of the underlying knockout trees to which the player is assigned, see Figure~\ref{fig:tree+seedings} for an example of a single knockout tournament. 

\begin{figure}[!h]
    \centering
    \begin{tikzpicture}
        \brackets{8}{1}{0,1,4,5,2,3,6,7}
    \end{tikzpicture}
    \caption{A single knockout $T$ where players $0,1, \ldots,7$ are assigned to the leaf nodes, leading to the seeding $s=0145-2367$.}
    \label{fig:tree+seedings}
\end{figure}
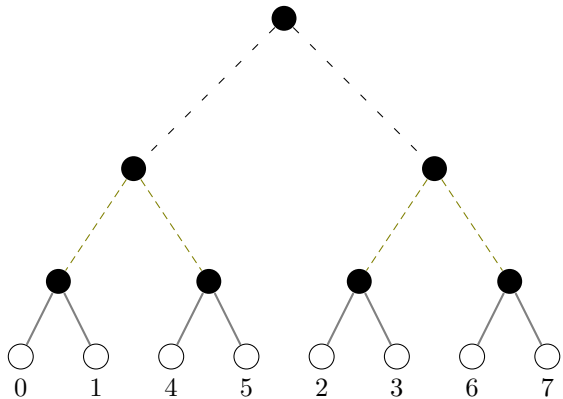

Once the seedings are specified, the individual knockout tournaments of the SKC can unfold - no other decisions in the design of the competition need to be taken. We refer to specifying the seedings as the {\em design} of the SKC.

In this note, we do not deal with determining the winner of an SKC; instead, we focus on the question: how to design an SKC in a fair way?


Here, we interpret fair by asking for a design that (i) treats all players equal without any prior assumptions on the strenghts of the players, and (ii) each pair of players should meet equally often in each of the rounds of an SKC.


One could argue that simply picking random seedings leads to a fair SKC as each player, in expectation, meets each other player equally often. However, it is clear that due to the inherent variability of picking random seedings, a design is found that violates these conditions.

Thus, we aim to find seedings such that, over the SKC, each pair of players meets equally often in all rounds. Consider for instance the first round: as the SKC consists of $2^k-1$ knockout tournaments, each player plays $2^k-1$ first round matches. Hence, we want to find seedings such that each player meets each other player exactly once in a first round. More generally, the question is: do there exist seedings such that each pair of players meets equally often in each of the rounds of the SKC?

We capture this notion formally by defining the notion of {\em stability} of an SKC.


\begin{definition}
\label{def:level}
Given a knockout tournament $T$ for $n=2^k$ players, we say that $v_T(x,x') = i$ if players $x,x'$ can meet in round $i$ of that tournament, $i=1, \ldots, k$. 
\end{definition}

The phrase `can meet' in the above definition refers to the assumption that players $x$ and $x'$ win their matches in the rounds prior to their encounter. For instance, in Figure~\ref{fig:tree+seedings}, players 1 and 4 can meet in Round 2, while players 0 and 3 can meet in Round 3, the final.

Let us now formally define the concept of stability, where we use $\#S$ to denote the number of elements of a finite set $S$.

\begin{definition}
\label{def:stable}
Given a set of knockout tournaments $\mathcal{T}$ on $n=2^k$ players, we say that it is {\em stable in round $i$} if there is a number $c_i$  so that
$$\#\{T\in\mathcal{T}: v_T(x,x')=i\} = c_i$$ for all pairs of distinct players $x,x'$. 
We say that the set $\mathcal{T}$ is {\em stable} if it is stable in all rounds $i=1,\dots,k$.
\end{definition}

Observe that the expression $\#\{T\in\mathcal{T}: v_T(x,x')=i\}$ counts the tournaments $T$ from the set $\mathcal{T}$ such that players $x$ and $x'$ can meet at round $i$ in $T$, $1 \leq i \leq k$.

\begin{definition}
We define a Serial Knockout Competition (SKC) as a competition for $n=2^k$ players consisting of $n-1$ knockout tournaments. 
\end{definition}

Notice that in an individual knockout tournament $T$, a player can meet any of $2^{i-1}$ other players when reaching round $i$, i.e., for each player $x$, we have $\#\{x' : v_T(x,x') = i\} = 2^{i-1}$, $i=1, \ldots, k$. As an SKC consists of $2^k-1$ knockout tournaments, the number of meetings that are possible in round $i$ for any player is given by $(2^k-1)2^{i-1}$, $1 \leq i \leq k$. With the number of opponents of any player $x$ equal to $n-1 = 2^k-1$, an SKC is stable in round $i$ if $c_i = 2^{i-1}$, for $i=1, \ldots, k$.

In this note, we prove that stable SKC's exist for arbitrary $n=2^k$. We describe in Section~\ref{sec:motivation} the case that motivates this work. In Section~\ref{sec:n=8} we investigate the case of 8 players, and in Section~\ref{sec:general} we deal with the general case. We illustrate in Section~\ref{sec:1632} the cases of 16 and 32 players, and we close in Section~\ref{sec:conclusion}.

\subsection{Motivation: The Premier League of Darts}
\label{sec:motivation}

The motivation for investigating this particular tournament design comes from the Professional Darts Corporation (PDC). We now describe this competition in more detail.

The Premier League of Darts, organized by the PDC, is an annual competition where the best darts players of the world compete over several months for the title. This year's edition featured the best 8 players, started at February 3, 2022, and ended at June 13, 2022. Total prize money is £1.000.000, and the winner pockets £275.000. The concept of the league changed drastically compared to the previous years – this edition consists of 16 knockout tournaments. Thus, there is a winner for each of these knockout tournaments, and, importantly, in every single match there is something to play for, which adds to the excitement of the format. 

The 16 knockout tournaments are structured in the following way: the first 7 knockout tournaments have a predetermined seeding, then there is a special knockout tournament, again 7 knockout tournaments with a given seeding, and a last special knockout tournament. The seedings in the special knockout tournaments depend on the standings at that time. The other (regular) knockout tournaments have a fixed seeding that is determined in advance by the PDC. Our analysis focuses on the seedings in these regular knockout tournaments. The first 7 knockout tournaments, as well as the second 7 regular knockout tournaments, each correspond to an SKC. 

As far as we are aware, this is the first occurence of an SKC in practice. One reason explaining why an SKC format is not being used more often in practice is the fact that knockout tournaments are used when a match is physically (or otherwise) demanding, and one wants to have relatively few matches to determine a winner. As an SKC requires multiple knockouts, it does not constitute a format with few matches. However, this argument does not apply when the tournament can be organized over a relatively long time period (as in the case of the PDC), and it also does not apply in the domain of e-sports as these require little (physical) effort. E-sports are a fast growing domain with an enormous amount of competitions being organized. We expect that the format of an SKC, or variations thereof, will turn out to be useful and popular in e-sports, as it combines the excitement of a knockout format with the fairness of a round robin format.


\section{Constructing a stable SKC when $n=8$}
\label{sec:n=8}



In this section, we are going to construct a stable SKC tournament $\mathcal{T} = (T_r)_{r \leq 7}$ for $8$ players; this analysis applies directly to the situation encountered by the PDC (see Section~\ref{sec:motivation}). Each knockout tournament is specified by providing a \textit{seeding} $s$, i.e., an ordered permutation of the players $0, \ldots, 7$. In Figure~\ref{exa:seeding} it is shown how to make a knockout tree out of the seeding $s = 01452367$. Although the permutation itself holds all the information needed, we may place hyphens as a visual aid indicating the halves of the seeding:  $0145-2367$ instead of $01452367$.

\begin{example}
\label{exa:seeding}
The permutation $0145-2367$ corresponds to the tree in Figure~\ref{exa:seeding}.
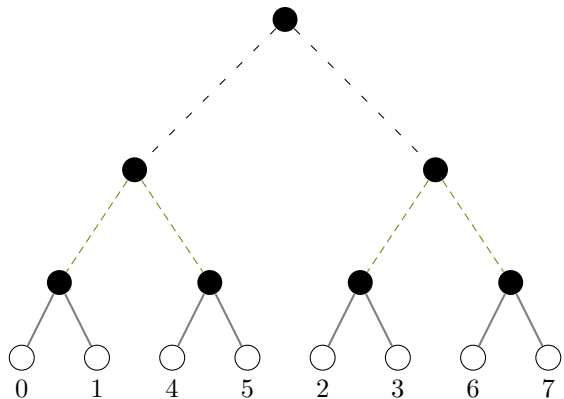
\begin{figure}[!h]
    \centering
    \begin{tikzpicture}
        \brackets{8}{1}{0,1,4,5,2,3,6,7}
    \end{tikzpicture}
    \caption{Knockout tree $T$ with seeding $s=0145-2367$.}
    \label{fig:seedings}
\end{figure}

\end{example}

As for the construction, we first simply state a stable SKC in Table \ref{tbl:stable8}, after which we give a method to generate such a set of seedings. 
\begin{table}[!h]
    \centering
    \begin{tabular}{|cccc|} 
    \hline
         Knockout & Seeding & Node & Line\\ 
         Tournament & & & \\ \hline
         1 & 0145-2367 & 1 & Red \\
         2 & 0426-1537 & 4 & Purple \\
         3 & 0213-4657 & 2 & Light green \\
         4 & 0356-1247 & 3 & Blue \\
         5 & 0527-1436 & 5 & Orange \\
         6 & 0734-1625 & 7 & Green \\
         7 & 0617-2435 & 6 & Light blue \\
    \hline
    \end{tabular}
    \caption{Seedings for a stable SKC.}
    \label{tbl:stable8}
\end{table}

In Table \ref{tbl:stable8}, the last two columns refer to nodes and lines. These nodes and lines are elements of the Fano-plane used to get to these seedings. This plane is depicted in Figure $\ref{fig:fanoplane}$, where the players $1$ to $7$ are placed on the seven nodes. We construct a seeding in the following way:

\begin{itemize}
    \item Select a node $x \in \{1,\ldots,7\}$. This indicates that Player $0$ meets Player $x$ in the first knockout tournament. In case $x=1$, we have a partial seeding $s=01\dots$.
    \item Select a line that goes through node $x$. The players corresponding to the two other nodes on the line meet each other. In case $x=1$, if we select the red line, then players $4,5$ meet and we extend the partial seeding to $s=0145\dots$.
    \item The remaining two matches are given by the two non-selected lines through node $x$. The two players on each line respectively, meet each other. This means that, in case $x=1$, players $2,3$ (light green) and $7,6$ (light blue) meet in the first knockout tournament. The resulting seeding for the first knockout tournament is thus given by $0145-2376$.
\end{itemize}

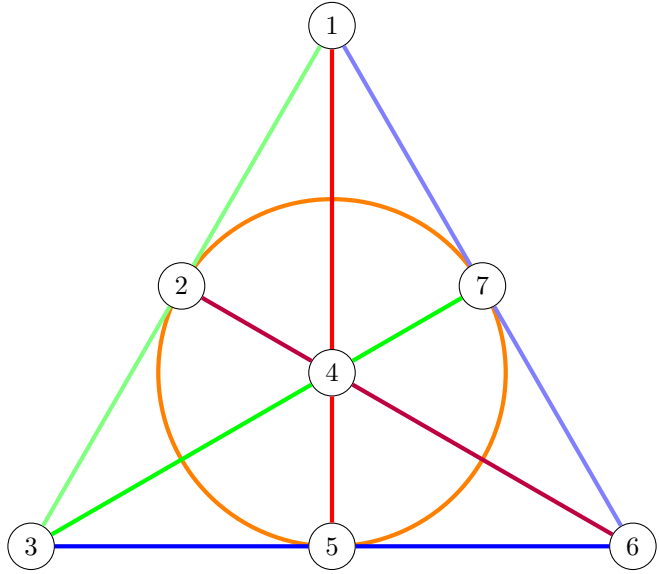
\begin{figure}[!h]
    \centering
    \begin{tikzpicture}
        \fano
    \end{tikzpicture}
    \caption{The Fano-plane used to construct \Cref{tbl:stable8}}
    \label{fig:fanoplane}
\end{figure}

A routine verification shows that the knockout tournament arising from a node and a line has the following key property.
\begin{lemma} \label{lem:level}Let $T$ be the knockout tournament that arises from the node-line pair $x,\ell$ of the Fano plane, and let $y$ be a node of the Fano plane. Then 
\begin{itemize}
\item $v_T(0,y)=1$ if and only if $y=x$,
\item $v_T(0,y)=2$ if and only if $y\in \ell$ and $y\neq x$, and
\item $v_T(0,y)=3$ if and only if $y\not\in \ell$.
\end{itemize}
Moreover, if $\ell'=\{y,x, x'\}$ is any line of the Fano plane containing the node $y$, then $v_T(x,x')=v_T(0,y)$.
\end{lemma}
Notice that in Table \ref{tbl:stable8}, each node and each line of the Fano plane occur exactly once, and each node is on the corresponding line. The following theorem states that this construction is sufficient to obtain a stable SKC.

\begin{theorem}
Let $x_1,\ldots, x_7$ be an enumeration of the nodes and $\ell_1,\ldots, \ell_7$ be an enumeration of the lines of the Fano plane, such that $x_r \in \ell_r$ for $r=1,\ldots, 7$. Let $T_r$ be the the knockout tournament that arises from the the pair $x_r, \ell_r$. Then, the SKC defined by
${\mathcal T}:=\{T_1,\ldots, T_7\}$
is stable.
\end{theorem}
\begin{proof} 
To show that $\mathcal{T}$ is stable, we need to show that 
\begin{equation}
\label{eq:theeq}
    \#\{T\in \mathcal{T}: v_T(x,x')=i\}=2^{i-1},
\end{equation}
for each pair of distinct players $x,x'$ and each round $i\in \{1,2,3\}$. Notice that $\mathcal{T}$ is stable in round $i=3$ if it is stable in both round $1$ and $2$.

We first consider the case that one of $x,x'$ is 0, say $\{x, x'\}=\{0,y\}$ for some $y\in \{1,\ldots, 7\}$.

\begin{itemize}
\item When $i=1$, our construction ensures that in each individual knockout tournament $r=r^y$, there exists a unique player $x_r=y$ meeting player 0. Hence, $\#\{T\in \mathcal{T}: v_T(0,y)=1\}=\#\{r: y=x_r\}=1$, and equation (\ref{eq:theeq}) is satisfied for $i=1$.

 

\item When $i=2$, we observe that there are exactly three lines through $y$, thus there exist two distinct knockout tournaments $r,r' \neq r^y$ such that $y \in \ell_r,\ell_r'$ - meaning that $(0,y)$ can meet in round $2$ in those knockout tournaments. Thus: $\#\{T\in \mathcal{T}: v_T(0,y)=2\}=\#\{r: y\in \ell_r, y\neq x_r\}=2$, and equation (\ref{eq:theeq}) is satisfied for $i=2$.
\end{itemize}
This settles the case where one player is Player 0. Next, suppose $x,x'$ are distinct players, both not $0$. Then, the Fano plane contains a unique node $y$ and line $\ell' = \{y,x,x'\}$ through $x,x'$. By Lemma \ref{lem:level}, we have $v_T(x,x')=v_T(0,y)$ for each $T\in \mathcal{T}$. As $\#\{T \in \mathcal{T} : v_T(0,y) = i\} = 2^{i-1}$ for all $y$, this holds for any distinct pair $x,x'$, for $i=1,2,3$.

The theorem follows.
\end{proof}

We point out that, from the viewpoint of stability, the sequence with which the individual knockout tournaments are played, is irrelevant.

\section{Constructing a stable SKC}
\label{sec:general}
Here we generalize the node-line construction used in Section~\ref{sec:n=8} to find a stable SKC for $n=2^k$ players. In Section~\ref{sec:keyidea}, we describe the basic idea, and in Section~\ref{sec:galois} we make a connection to Galois fields. We use this connection in Section~\ref{sec:result} to prove our main result: Theorem~\ref{th:main}.

\subsection{The basic idea}
\label{sec:keyidea}
The key idea that we will carry over to the general setting, is that we will construct our knockout tournaments in a restricted way, so that for each pair of players $x,x'$, there is a well-defined player $y$ such that $$v_T(x,x')=v_T(0,y)$$
for all knockout tournaments $T$ of this restricted form. 
Showing that an SKC $\mathcal{T}$ is stable, where each tournament $T\in \mathcal{T}$ is of this special form, then reduces to verifying that
$$\#\{T\in \mathcal{T}: v_T(0,y)\}=2^{i-1}$$
for each player $y$ and each round $i$, $i=1, \ldots,k$.

To define the representative $y$ of a pair of players $x,x'$ and to create the special tournaments $T$, we need additional structure on the set of players. For the case $n=8$, we identified the non-zero players with nodes of the Fano plane and used its geometry to define the tournaments. In what follows, we will identify the $n=2^k$ players with the $2^k$ elements of the {\em Galois field} $GF(2^k)$. 

As $GF(2^k)$ is a field, both addition and multiplication are possible operations on its elements. We construct a tournament $T$ such that for $x,x' \in GF(2^k)$, we have \begin{align*}v_T(x,x')=v_T(0,y)\end{align*}
when $y:=x-x'$. 

After we have constructed a base model for our knockout tournament, we use the multiplication in $GF(2^k)$ on $T$, to create tournaments $T(z)$ for each nonzero element $z$ of $GF(2^k)$, and argue that $$\mathcal{T}:=\{T(z): z\neq 0\}$$ is a stable SKC.

\subsection{The connection to Galois fields}
\label{sec:galois}
To exploit the structure of Galois field $GF(2^k)$, we first have to describe $GF(2^k)$. Although we do not go into too much detail, we point out the main properties that we use. For an accessible introduction to finite fields,  see \cite{FF2022}.\\

\noindent A {\em binary polynomial} $q\in\mathbb{Z}_2[X]$ is an expression of the form $$q = q_kX^k + \dots q_1X + q_0$$
where the coefficients $q_i$ are either $0$ or $1$. Such polynomials may be added and multiplied as usual, but taking into account that the coefficients are added according to the rule $1+1=0$. So e.g.
$$(X+1)\cdot(X^2+X+1)=X^3+X^2+X^2+X+ X+1=X^3+1$$

The degree of a polynomial  $q=\sum_i q_i x^i$ is the highest value of $i$ so that $q_i\neq 0$.
The polynomial $q=X^3+1$ that is the outcome of the above calculation is {\em reducible}, because it has degree 3 and is the product of two polynomials of strictly lower degree, resp. $X+1$ of degree 1 and $X^2+X+1$ of degree 2. 
For any value of $k$, irreducible  polynomials $q\in\mathbb{Z}_2[X]$ are guaranteed to exist. For example, when $k=3$, the polynomial $q = X^3 + X^2 + 1$ is irreducible over $\mathbb{Z}_2[X]$. Other irreducible polynomials of small degree are $X^2+X+1, X^4+X+1, X^5+X^2+1$ for degree $k=2,4,5$ respectively.
  
Given any polynomial $q\in \mathbb{Z}_2[X]$ , we write $\mathbb{Z}_2[X]/(q)$ for the set of polynomials one gets from a polynomial in $\mathbb{Z}[X]$ by filling in a symbolic value $\alpha$ that is assumed to satisfy $q(\alpha)=0$. 
If $q=X^2+X+1$, then the element $x=\alpha^3\in\mathbb{Z}_2[X]$ can be rewritten as
$$x=\alpha^3=\alpha^3+\alpha\cdot q(\alpha)=\alpha^3+\alpha\cdot(\alpha^2+\alpha+1)=\alpha^2 +\alpha=\alpha^2 +\alpha+q(\alpha)=1$$
because $q(\alpha)=0$. Indeed, any element $x\in \mathbb{Z}_2[X]/(q)$ can be rewritten to $x= x_{k-1} \alpha^{k-1}+\cdots+x_1\alpha+x_0$, that is, without using powers $\alpha^i$ with $i\geq k$ in the expression.

If $q \in \mathbb{Z}_2[X]$  is an  {\em irreducible} polynomial of degree $k$, it is known that $GF(2^k) \cong \mathbb{Z}_2[X]/(q)$ is a {\em field}: one can add and multiply with its elements, but also divide by any nonzero element. Indeed, consider that in the above example with $q=X^2+X+1$, we had $\alpha\cdot \alpha^2=\alpha^3=1$. Then $\alpha^{-1}=\alpha^2$, and dividing by $\alpha$ amounts to multiplying with $\alpha^2$. The irreducibility of $q$ ensures that for any nonzero $x\in GF(2^k)$ there is a $y\in GF(2^k)$ so that $x\cdot y=1$. Then a division by $x$ can be executed as a multiplication by $y$.

There is more than one irreducible polynomial $q$ of each degree $k$, but whichever one uses, the outcome is mathematically `the same` field $GF(2^k)$. Having fixed a polynomial $q$ for the construction of the Galois field $GF(2^k)$, there is just one way to write an element $x\in GF(2^k)$ as
$x = \sum_{i=0}^{k-1} x_i\alpha^i \in GF(2^k)$, and we may define the {\em degree} of $x$ as $d(x) = \max\{i : x_i \neq 0\}$. 

This degree leads us to the following lemma on the existence of a tournament $T$ with the nice property that $v_T(x,y) = v_T(0,x-y) = 1 + d(x-y)$.

\begin{lemma} There is a knockout tournament $T$ whose players are the elements of $GF(2^k)$, so that $v_T(x,y)=1+d(x-y)$ for all $x,y\in GF(2^k)$. 
\label{lem:knockout}
\end{lemma}
\begin{proof} 
We construct tournament $T$ by inductively constructing $T_m$ for incremental values $m=1,\dots,k$, where each $T_m$ is a knockout tournament on the set $P_m = \{x \in GF(2^k) : d(x) < m\}$, and all the $T_m$ have the property that $v_{T_m}(x,y) = 1 + d(x-y)$ for $x,y \in P_m$. Then $T=T_k$ proves the lemma.

When $m=1$, the set $P_0=\{0,1\}$ contains only two players, and the unique tournament $T_1$ one can construct on these two players has $v_{T_1}(0,1)=1=1+d(1-0)$.

As induction step, assume that $T_m$ exists such that $v_{T_m}(x,y)=1+d(x-y)$ for all $x,y\in P_m$. Let $T_m'$ arise from a copy of $T_m$ by adding $\alpha^m$ to each player. Then $T_m'$ has players $P_m'=\{x+\alpha^m: x\in P_m\}$ and for any two players $x',y' \in P'_m$ we have
\begin{align*}v_{T_m'}(x', y')=v_{T_m}(x,y)=1+d(x-y)=1+d(x'-y')\end{align*}
where $x'=x+\alpha^m$ and $y'=y+\alpha^m$ with $x,y\in P_m$.

We construct $T_{m+1}$ for players $P_{m+1}=P_m\cup P_m'$ as the combination of tournaments $T_m$, $T_m'$, where at round $m+1$, the winner of $T_m$ plays the winner of $T_m'$. For this $T_{m+1}$, we see that for $x,y \in P_{m+1}$:
\begin{align*}
    v_{T_{m+1}}(x,y) &= v_{T_{m}}(x,y) = 1 + d(x-y) && \text{if }x,y \in P_m \\
    v_{T_{m+1}}(x,y) &= v_{T_{m}'}(x,y) = 1 + d(x-y) &&\text{if } x,y \in P_m' \\
    v_{T_{m+1}}(x,y) &= 1 + m = 1 + d(x-y) &&\text{if } x \in P_m, y \in P_m' \text{ or } x \in P_m', y \in P_m
\end{align*}
This finishes the induction step. Taking $T = T_k$ gives the desired tournament.
\end{proof}

The construction of $T$ with elements in $GF(2^3)$ is given in Figure \ref{fig:ko-construction}.
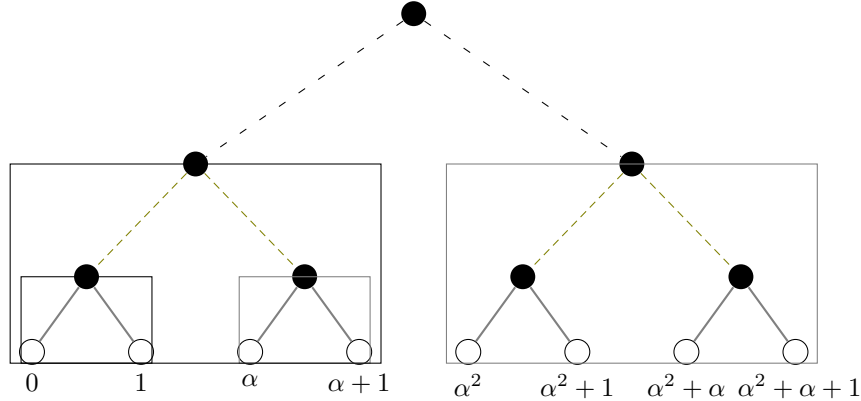
\begin{figure}[!h]
    \centering
        \begin{tikzpicture}
        \brackets{8}{1.45}{0,1,\alpha,\alpha+1,\alpha^2,\alpha^2+1,\alpha^2+\alpha,~\alpha^2+\alpha+1}
        \draw[draw=black] (0.9*1.45,-0.15) rectangle (2.1*1.45,1);
        \draw[draw=gray] (2.9*1.45,-0.15) rectangle (4.1*1.45,1);
        \draw[draw=black] (0.8*1.45,-0.15) rectangle (4.2*1.45,2.5);
        \draw[draw=gray] (4.8*1.45,-0.15) rectangle (8.2*1.45,2.5);
    \end{tikzpicture}
\caption{A knock-out tournament $T$ so that $v_T(x,y)=1+d(x-y)$}
\label{fig:ko-construction}
\end{figure}

\subsection{The result}
\label{sec:result}

By Lemma~\ref{lem:knockout}, we know there exists a knockout tournament $T$ on the elements of $GF(2^k)$ such that $v_T(x,y) = v_{T}(0,x-y) = 1+d(x,y)$ for all $x,y \in GF(2^k)$. In the following section, we argue that for each non-zero $z \in GF(2^k)$, the tournament $T(z)$ obtained from $T$ by replacing each player $x$ by $zx$ maintains the property that $v_{T(z)}(x,y) = v_{T(z)}(0,x-y)$. Then we show that $$\mathcal{T} = \{T(z) : z \in GF(2^k) \setminus \{0\}\}$$ is a stable SKC.



Let $T$ be a tournament satisfying Lemma~\ref{lem:knockout}, thus $v_T(x,y) = 1 + d(x-y)$ for all $x,y\in GF(2^k)$. Let $z\in GF(2^k)$ be non-zero and thus invertible. We construct $T(z)$ from $T$ by replacing each player $x$ with $zx$. As the map $x \mapsto zx$ is one-to-one, $T(z)$ is again a tournament whose players are the elements of $GF(2^k)$. Evidently we have $v_{T(z)}(x, y)=v_T(z^{-1}x, z^{-1}y)$ for all $x,y\in GF(2^k)$. It follows that 
$$v_{T(z)}(x, y)=v_T(z^{-1}x, z^{-1}y)=v_T(0, z^{-1}(x-y))=v_{T(z)}(0,x-y)$$
for all $x,y\in GF(2^k)$ and 
$$v_{T(z)}(0,y)=v_{T}(0,z^{-1}y)=1+d(z^{-1}y)$$
for all $y\in GF(2^k)$.
\begin{theorem}
\label{th:main}$\mathcal{T}:=\{T(z): z\text{ a nonzero element of } GF(2^k)\}$
is a stable SKC.
\end{theorem}
\begin{proof} We need to show that $\#\{T\in \mathcal{T}: v_T(x,x')=i\}=2^i$
for each pair of distinct players $x,x'\in GF(2^k)$  and each round $i=1,\ldots, k$.

If one of $x,x'$ is 0, say $\{x,x'\}=\{0,y\}$ with $y\neq 0$, then, for each $i=1, \ldots,k$,
\begin{align*}\#\{T\in \mathcal{T}: v_T(0,y)=i\}= \#\{z\in GF(2^k): z\neq 0, 1+d(z^{-1}y)=i\}.\end{align*} 
Substituting $z$ by $r^{-1}y$ this equals 
\begin{align*}&\#\{r^{-1}y\in GF(2^k): r\neq 0, 1+d(r)=i\}=\\
&\#\{r\in GF(2^k): r\neq 0, 1+d(r)=i\}=2^i\end{align*}
since the map $r\mapsto r^{-1}y$ is one-to-one. 

The general case reduces to the above special case, since each of the tournaments $T\in \mathcal{T}$ has $v_T(x,x')= v_T(0, x-x')$. Then
\begin{align*}\#\{T\in \mathcal{T}: v_T(x,x')=i\}=\#\{T\in \mathcal{T}: v_T(0, x-x')=i\}=2^i,\end{align*}
as required. \end{proof}

\ignore{
For example, if $k=3$, we have a tournament with the following matches.
\begin{itemize}
\item at stage 0: 0 vs 1; $\alpha$ vs. $\alpha+1$; $\alpha^2$ vs. $\alpha^2+1$; $\alpha^2+\alpha$ vs. $\alpha^2+\alpha+1$
\item at stage 1: winner of 0,1 vs. winner of $\alpha$, $\alpha+1$; winner of $\alpha^2$,$\alpha^2+1$ vs. winner of $\alpha^2+\alpha$, $\alpha^2+\alpha+1$
\item at stage 2: winner of 0,1, $\alpha$, $\alpha+1$ vs.  winner of $\alpha^2$,$\alpha^2+1$, $\alpha^2+\alpha$, $\alpha^2+\alpha+1$
\end{itemize}
}

We close this section with an example that constructs a stable SKC on $8$ players using the Galois group. 

\begin{example}
For the Galois group, we choose $q(X)=X^3+X+1$ as the irreducible polynomial over $\mathbb{Z}_2$ and set $q(\alpha) = 0$. The corresponding multiplication table is shown in Table \ref{tab:multiplic}.

\begin{table}[!h]
    \centering
    \tiny
    \begin{tabular}{|C*{7}{C}|} \hline 
         & 1 & \alpha & \alpha + 1 & \alpha^2 & \alpha^2+1 & \alpha^2 + \alpha & \alpha^2 + \alpha + 1\\ \hline
        0 & 0 & 0 & 0 & 0 & 0 & 0 & 0 \\ \hline
         1 & 1 & \alpha & \alpha + 1 & \alpha^2 & \alpha^2+1 & \alpha^2 + \alpha & \alpha^2 + \alpha + 1 \\ \hline
         \alpha & \alpha & \alpha^2 & \alpha^2 + \alpha & \alpha+1 & 1 & \alpha^2 + \alpha+1 & \alpha^2 + 1 \\ \hline
         \alpha + 1 & \alpha + 1 & \alpha^2 + \alpha & \alpha^2 + 1 & \alpha^2 + \alpha + 1 & \alpha^2 & 1 & \alpha \\ \hline
         \alpha^2 & \alpha^2 & \alpha + 1 & \alpha^2 + \alpha + 1 & \alpha^2+\alpha & \alpha & \alpha^2+1 & 1 \\ \hline
         \alpha^2 + 1 & \alpha^2 + 1 & 1 & \alpha^2 & \alpha & \alpha^2+\alpha+1 & \alpha + 1 & \alpha^2 + \alpha \\ \hline
         \alpha^2 + \alpha & \alpha^2 + \alpha & \alpha^2 + \alpha+1 & 1 & \alpha^2+1 & \alpha + 1 & \alpha & \alpha^2 \\ \hline
         \alpha^2 + \alpha + 1 & \alpha^2 + \alpha + 1 & \alpha^2+1 & \alpha & 1 & \alpha^2+\alpha & \alpha^2 & \alpha+1 \\ \hline
    \end{tabular}
    \caption{Multiplication on $GF(2^3)$}
    \label{tab:multiplic}
\end{table}

Table \ref{tab:multiplic} essentially gives the seedings for the SKC, since the row for multiplication by $z$ presents the seeding for $T(z)$. Upon replacing each polynomial with the number specified in \Cref{tab:galteams}, we get the SKC of Table \ref{tab:GaloisSKO}.

\begin{table}
\centering
    \begin{tabular}{|*{8}{C}|}  \hline 
         0 & 1 & \alpha & \alpha + 1 & \alpha^2 & \alpha^2+1 & \alpha^2 + \alpha & \alpha^2 + \alpha + 1\\ \hline
         0&1&4&5&2&3&6&7\\
         \hline
         \end{tabular}
         \caption{From Galois to teams}
         \label{tab:galteams}
\end{table}

\ignore{Take any permutation of $1,\dots,7$, and let the $i$-th element in the permutation correspond to the $i$-th element in the first row. Apply this correspondence to all of the values in the Table, to get $7$ rows with a permutation of $1,\dots,7$. Starting with the $0$ player and reading each of these $7$ rows, gives the seeding of $7$ knock-out tournaments, that together form a stable SKC. 

If we let the first row correspond to the first row of Table \ref{tbl:stable8}, that is $145-2367$, omitting the $0$ player, than the knockout schemes can be read as follows:
}
Comparing the SKC from Table \ref{tbl:stable8} with the one shown in Table \ref{tab:GaloisSKO}, we see that the knockout tournaments are the same and merely permuted.
\begin{table}[!h]
    \centering
\begin{tabular}{|CCCC|}
\hline
   \begin{array}{c}\text{Knockout} \\ \text{Tournament}\end{array} & \text{Seeding} & \begin{array}{c}\text{Knockout} \\  \text{Tournament}\end{array} & \text{Seeding} \\ \hline
     1 & 0145-2367 & 5 & 0312-4756\\
     2 & 0426-5173 & 6 & 0671-3542 \\
     3 & 0563-7214 & 7 & 0734-1625 \\
     4 & 0257-6431 & & \\ \hline
\end{tabular}
    \caption{SKC constructed from \Cref{tab:multiplic}}
    \label{tab:GaloisSKO}
\end{table}
\end{example}

\section{Stable SKC on 16 and 32 players}
\label{sec:1632}
In this section we use the construction of the previous section to generate an SKC on $16$ and one on $32$ players. For notational purposes, we enumerate the first $10$ players by $0,\dots,9$ and continue with $a,b$ up until $f$ in the case of $16$ and $v$ in the case of $32$ teams. By doing this, we can visualize the seedings as a string of length $16$ ($32$) where each character is one player.

The seedings are shown in \Cref{tab:SKC16,tab:SKC32}.

\begin{table}[!h]
    \centering
    \begin{tabular}{|CC|} \hline
         \begin{array}{c}\text{Knockout} \\ \text{Tournament}\end{array} & \text{Seeding}  \\ \hline
         1 & 0123-4567-89ab-cdef \\
         2 & 0246-8ace-3175-b9fd \\
         3 & 0365-cfa9-b8de-7412 \\
         \hline
         4 & 048c-37bf-62ea-51d9 \\
         5 & 05a\f-72d8-eb41-9c36 \\
         6 & 06ca-bd71-539f-e824 \\
         7 & 07e9-\f816-da34-25cb \\
         \hline
         8 & 083b-6e5d-c4f7-a291 \\
         9 & 0918-2b3a-4d5c-6f7e \\
         10 & 0a7d-e493-f582-1b6c \\
         11 & 0b5e-a1\f 4-7c29-d683 \\
         \hline
         12 & 0cb7-59e2-a61d-f348 \\
         13 & 0d94-1c85-2fb6-3da7 \\
         14 & 0ef1-d32b-9768-4ab5 \\
         15 & 0fd2-964b-1ec3-875a \\ \hline
    \end{tabular}
    \caption{Balanced SKC on 16 players}
    \label{tab:SKC16}
\end{table}

\renewcommand{\arraystretch}{0.8}
\begin{table}[!h]
    \centering
    \sffamily
    \begin{tabular}{|CC|} \hline
         \begin{array}{c}\text{Knockout} \\ \text{Tournament}\end{array} & \text{Seeding}  \\ \hline
1 & 0123-4567-89ab-cdef-ghij-klmn-opqr-stuv\\
2 & 0246-8ace-gikm-oqsu-5713-df9b-lnhj-tvpr\\
3 & 0365-cfa9-orut-knih-lmjg-pqvs-deb8-1274\\ \hline
4 & 048c-gkos-51d9-lhtp-ae26-quim-fb73-vrnj\\
5 & 05af-khur-d872-psjm-qvgl-eb41-nito-369c\\
6 & 06ca-ouki-ljpv-db17-f935-nhrt-qsmg-24e8\\
7 & 07e9-sril-tqjk-16f8-vohm-34da-25cb-upgn\\ \hline
8 & 08go-5dlt-a2qi-f7vn-ks4c-hp19-ume6-rjb3\\
9 & 09ir-18jq-2bgp-3aho-4dmv-5cnu-6fkt-7els\\
10 & 0aku-d7pj-qge4-nt39-hr5f-sm82-b1vl-6cio\\
11 & 0bmt-92vk-ip4f-rgd6-1ans-83ul-jo5e-qhc7\\ \hline
12 & 0cok-lpd1-f3nr-qm2e-ui6a-b7jv-ht95-48sg\\
13 & 0dqn-hsb6-7atg-mrc1-e3kp-vi58-94ju-ol2f\\
14 & 0esi-tj1f-vh3d-2cug-rl79-68qk-4aom-pn5b\\
15 & 0fuh-pm78-no96-e1gv-b4lq-itc3-sj2d-5ark\\ \hline
16 & 0g5l-aqfv-k4h1-uerb-dt8o-7n2i-p9sc-j3m6\\
17 & 0h7m-ev9o-sdra-i3l4-tcqb-j2k5-1g6n-fu8p\\
18 & 0i1j-2g3h-4m5n-6k7l-8q9r-aobp-cudv-esft\\
19 & 0j3g-6l5m-cvfs-ap9q-obr8-udte-k7n4-i1h2\\ \hline
20 & 0kdp-qen3-h5s8-bv6i-7jau-t9g4-m2rf-co1l\\
21 & 0lfq-ubh4-pcm3-7i8t-n2od-9s6j-er1k-g5va\\
22 & 0m9v-i4rd-1n8u-j5qc-2kbt-g6pf-3las-h7oe\\
23 & 0nbs-m1ta-9u2l-v8k3-i5pe-4jfo-rcg7-dq6h\\ \hline
24 & 0old-fnq2-u6bj-h94s-p1ck-me3r-7via-8gt5\\
25 & 0pne-bis5-mf1o-t4aj-9gu7-2rlc-v68h-kd3q\\
26 & 0qhb-7tmc-ekv5-9jo2-s6dn-r1ag-i83p-lf4u\\
27 & 0rj8-3ogb-6tle-5umd-cnv4-fks7-ahp2-9iq1\\ \hline
28 & 0st1-v32u-r76q-4op5-jfei-cghd-8kl9-nbam\\
29 & 0tv2-r64p-jech-8lna-3us1-o57q-gdfi-bmk9\\
30 & 0up7-n9eg-blic-s25r-m8fh-1vo6-t34q-akjd\\
31 & 0vr4-jc8n-3so7-gfbk-6pt2-laeh-5qu1-m9di\\ \hline
  \end{tabular}
\caption{Balanced SKC on 32 players}
\label{tab:SKC32}
\end{table}

\newpage

\section{Discussion}
\label{sec:conclusion}

We have analyzed a novel tournament design that is used in practice, and that can be seen as a combination of a knockout tournament and a round robin tournament; we call it a Serial Knockout Competition (SKC). From the viewpoint of fairness an attractive property of an SKC is stability: whether or not pairs of players can meet equally often in the rounds of the SKC. We have shown that this is always possible. Interestingly, one easily observes that the implementation of the SKC used in the PDC Premier League is not stable.  

We remark here that the construction to create stable SKC's does not generate a unique tournament - for example, the order of the individual knockout tournaments can be changed without impacting the stability of the SKC. Also, within each knockout tournament, a tournament $T(s)$ with seeding $s$ can be replaced by $T(s')$ as long as $v_{T(s)}=v_{T(s')}$. Thus, not all stable SKC's are equal and from an organizer's point of view, there might be additional constraints allowing one to prefer one stable SKC over another.


\bigskip
\textbf{Acknowledgement}{The research of Frits C.R. Spieksma was partly funded by the NWO Gravitation Project NETWORKS, Grant Number 024.002.003.}
\newpage

\printbibliography

\end{document}